\documentclass[11pt,letterpaper]{article}
\usepackage[T1]{fontenc}
\usepackage[utf8]{inputenc}
\usepackage{lmodern}
\usepackage{amsmath,amssymb,amsthm,mathtools}
\usepackage[margin=1in]{geometry}
\usepackage{microtype,graphicx,enumitem}
\usepackage[section]{placeins}
\usepackage[numbers,sort&compress]{natbib}
\usepackage{xcolor}
\usepackage{tikz}
\definecolor{setblue}{RGB}{35,85,170}
\definecolor{setred}{RGB}{190,45,55}
\definecolor{forestgreen}{RGB}{35,125,75}
\usetikzlibrary{calc,positioning,decorations.pathreplacing,arrows.meta}
\usepackage{hyperref}
\hypersetup{hidelinks,
  pdftitle={Improved Bounds for the Maximum Reduced Denominator of Fractional Chromatic Numbers},
  pdfauthor={Xinbo Xu and Xuding Zhu}}
\newtheorem{theorem}{Theorem}[section]
\newtheorem{lemma}[theorem]{Lemma}
\newtheorem{proposition}[theorem]{Proposition}

\newtheorem{observation}[theorem]{Observation}
\theoremstyle{definition}
\newtheorem{definition}[theorem]{Definition}
\theoremstyle{remark}

\numberwithin{equation}{section}
\newcommand{\chif}{\chi_f}
\newcommand{\den}{\operatorname{den}}
\newcommand{\sd}{\operatorname{sd}}
\newcommand{\doi}[1]{\href{https://doi.org/#1}{doi:~\nolinkurl{#1}}}
\allowdisplaybreaks[1]
\title{Bounds on the maximum denominator of fractional chromatic numbers}
\author{Xinbo Xu \qquad Xuding Zhu}
\date{}
\begin{document}
\maketitle
\begin{abstract}
For a positive rational number $x=p/q$ with $p,q$ be coprime positive integers, let $\den(x) =q$.  For a positive integer $n$, let $\operatorname{sd}(n) = \max\{\den(\chi_f(G)): |V(G)|=n\}$. Let $D(n) = 2^{-n}(n+1)^{(n+1)/2}$. It was proved by Alon, Tuza and Voigt that    $\operatorname{sd}(n)\le D(n)$. In this paper, we prove that   
\[
\operatorname{sd}(n)\le D(n)n^{-\frac12\log_2 n+\frac14\log_2\log_2 n+O(1)}
\]
for all sufficiently large $n$.
For the lower bound on $\operatorname{sd}(n)$, it was recently proved by Bonamy, Hylasov\'a, Kaiser and Sereni that $\limsup_{n\to\infty}\sd(n)^{1/n}\ge\sqrt2.$
We prove that
\[
\limsup_{n\to\infty}\operatorname{sd}(n)^{1/n}\ge 3^{1/3} > \sqrt{2}.
\]

\end{abstract}
\medskip

\section{Introduction}\label{sec:introduction}

All graphs considered in this paper are finite and simple. A fractional colouring of a graph $G$ assigns nonnegative weights to its independent sets so that, for each vertex, the weights of the sets containing it sum to at least $1$. The fractional chromatic number $\chif(G)$ is the minimum total weight of a fractional colouring of $G$. This minimum is rational, since it is the optimum of a finite linear programme with rational coefficients. For a positive rational number $x=p/q$ with $p,q$ be coprime positive integers, let $\den(x) =q$. For a positive integer $n$, let
\[
\sd(n)=\max\{\den(\chif(G)):|V(G)|=n\}.
\]

Erd\H{o}s asked whether $\sd(n)\le n$ for every positive integer $n$. Chv\'atal, Garey and Johnson~\cite{ChvatalGareyJohnson1978} answered this question negatively by constructing graphs that give a superpolynomial lower bound on $\sd(n)$. On the other hand, by using Cramer's rule and Hadamard's inequality, they proved  that $\sd(n)\le n^{n/2}$. Alon, Tuza and Voigt~\cite[Lemma 2.1]{AlonTuzaVoigt1997} improved the upper bound to
\begin{equation}\label{eq:ATV-intro}
\sd(n)\le D(n),\qquad D(n)=2^{-n}(n+1)^{(n+1)/2}.
\end{equation}
Larsen, Propp and Ullman~\cite{LarsenProppUllman1995}
proved that
\[
 \chif(M(G))=\chif(G)+\frac{1}{\chif(G)},
\]
where $M(G)$ is the Mycielskian of $G$.
Using this formula, Fisher~\cite{Fisher1995} showed that
repeatedly applying the Mycielski construction to $K_3$
produces graphs $G$ with $\den(\chi_f(G))$
  growing exponentially in the number
of vertices of $G$.   More recently, Bonamy, Hylasov\'a, Kaiser and Sereni~\cite{BonamyHylasovaKaiserSereni2026} constructed tower graphs and proved that
\begin{equation}\label{eq:old-lower-intro}
 \limsup_{n\to\infty}\sd(n)^{1/n}\ge\sqrt2.
\end{equation}

Throughout, $\log$ denotes the natural logarithm and $\log_2$ the base-two logarithm.
In this paper, we improve both the upper bound~\eqref{eq:ATV-intro} and the   lower bound~\eqref{eq:old-lower-intro}.

\begin{theorem}\label{thm:upper-main}
For all sufficiently large integers $n$,
\begin{equation}\label{eq:upper-main-final}
\sd(n)\le D(n)n^{-\frac12\log_2 n+
\frac14\log_2\log_2 n+O(1)}.
\end{equation}
\end{theorem}

The proof uses the independent-set structure of the columns in the fractional-colouring programme. Let $C$ be the incidence matrix of the independent sets of $G$, with rows indexed by vertices and columns indexed by independent sets.
Two rows of $C$ indexed by  adjacent vertices cannot both have entry $1$ in the same column. We show that if an $n$-vertex graph $G$ has a matching of size $s$, then
\[
 \den(\chif(G))\le\gamma^sD(n).
\]
For fractionally vertex-critical graphs with a small maximum matching, Sperner's theorem bounds the number of vertices in terms of the matching number. Combining this bound with
the determinant estimate gives Theorem~\ref{thm:upper-main}. Sperner's theorem bounds the order of this subgraph in terms of its matching number. Combining these estimates gives Theorem~\ref{thm:upper-main}.

In~\cite{BonamyHylasovaKaiserSereni2026}, Bonamy et al. defined the tower graph.  The \emph{tensor product} $G \times H$ of $G$ and $H$ has vertex set $\{(x,y): x \in V(G), y \in V(H)\}$, in which $(x,y) \sim (x',y')$ if $xx' \in E(G)$ and $yy' \in E(H)$. 

\begin{definition}
    \label{def-tower}
    Let $H$ be a graph, and let $k\ge\ell\ge1$ be integers. The tower graph $T_{k,\ell}(H)$ is constructed as follows:

\begin{enumerate}
    \item Let $P_{k+\ell}^{\,\ell}$ denote the $\ell$th power of the path with vertices $p_{1-\ell},p_{2-\ell},\ldots,p_k$.
    \item 
    Take the tensor product $P_{k+\ell}^\ell \times H$, denotes its by $p_{i,h}$, with $h\in V(H)$. Hence  
\[
 p_{i,h}p_{j,h'}\in E
 \quad\Longleftrightarrow\quad
 1\le|i-j|\le\ell\ \text{ and }\ hh'\in E(H).
\]
\item For each $i\in\{1-\ell,\ldots,0\}$, identify all vertices of the fibre  $\{p_{i,h}:h\in V(H)\}$ into one vertex $z_{i+\ell}$, suppressing any parallel edges created by the identification.
\item Add the edges of a copy of $H$ inside the terminal fibre $\{p_{k,h}:h\in V(H)\}$.
\end{enumerate}     
  
\end{definition}

In ~\cite{BonamyHylasovaKaiserSereni2026}, the fractional chromatic numbers of $\chif(T_{k,\ell}(K_2))$ for arbitrary $k$ and $\ell$ are determined. It follows from this result that  
\[
 \limsup_{n\to\infty}\sd(n)^{1/n}\ge \sqrt{2}.
\]
In this paper, we  determine $\chif(T_{k,\ell}(K_m))$ for $m\ge2$.

\begin{theorem} \label{thm:chif}
For fixed integers $m\ge2$ and $\ell\ge1$, for $k \ge 2$, let 
 $Q_k$ be integers defined recursively as follows: 
\[
 Q_k=mQ_{k-
 1}-(m-1)^\ell Q_{k-\ell-1}
 \qquad(k\ge2),
\]
with $Q_t=0$ for $t\le0$ and $Q_1=1$.
We prove that
\[
 \chif(T_{k,\ell}(K_m))=\ell+1+\frac{(m-1)^k}{Q_k}
 \qquad(k\ge\ell),
\]
and  $\gcd( Q_k, m-1)=1 $, and hence $\den(\chi_f( T_{k,\ell}(K_m))=Q_k$. 
\end{theorem}

The proof extends the fractional colouring and fractional clique constructed in~\cite{BonamyHylasovaKaiserSereni2026}. For each real number $c$ with $m-1<c<m$, the recurrence gives $Q_k\ge c^{k-1}$ for all $k\ge1$, provided $\ell$ is sufficiently large. Since $T_{k,\ell}(K_m)$ has $mk+\ell$ vertices, this yields  the following result:

\begin{theorem}\label{thm:asymptotic}
For every integer $m\ge2$,
\[
 \limsup_{n\to\infty}\sd(n)^{1/n}\ge m^{1/m}.
\]
Consequently,
\[
 \limsup_{n\to\infty}\sd(n)^{1/n}
 \ge 3^{1/3}>\sqrt2.
\]
\end{theorem}

\section{Proof of the upper bound}\label{sec:upper}

Throughout this section, let $G$ have $n\ge2$ vertices,
and write $q(G)=\den(\chif(G))$.
We first show how vertex-disjoint stars improve the determinant bound for
$q(G)$. A large matching already gives the required improvement. When the
maximum matching is small, we construct a spanning star forest with
component orders $O(n/\log n)$.

\subsection{The determinant estimate from stars}
\label{subsec:upper-determinants}

Fix an ordering of $V(G)$ and list its nonempty independent sets as
$I_1,\ldots,I_M$. Let $C\in\{0,1\}^{n\times M}$ be their incidence matrix:
its row indexed by $v$ has entry $C_{v,j}=1$ precisely when $v\in I_j$.
We write $\mathbf 0_d$ and $\mathbf 1_d$ for the zero and all-one column
vectors of length $d$, and $\mathrm{Id}_d$ for the identity matrix of
order $d$.

\begin{lemma}\label{lem:upper-basis}
There is a nonsingular matrix $A\in\{0,1\}^{n\times n}$ whose columns
are indicators of independent sets of $G$ and for which
$q(G)\mid |\det A|$.
\end{lemma}

\begin{proof}
Introducing the excess vector $\xi=Cx-\mathbf 1_n$ gives
\begin{equation}\label{eq:upper-lp}
\chif(G)=
\min\left\{
\mathbf 1_M^{\mathsf T}x:
E\begin{pmatrix}x\\\xi\end{pmatrix}=\mathbf 1_n,
\quad x,\xi\ge0
\right\},
\qquad E=[\,C\; -\mathrm{Id}_n\,].
\end{equation}
The matrix $E$ has full row rank, and the programme is feasible and bounded
below. The optimal-basis theorem therefore gives an optimal basis $H$, a
nonsingular $n\times n$ submatrix of $E$ with
$H^{-1}\mathbf 1_n\ge0$; see~\cite[Sections~7.4, 8.5 and 8.9]{Schrijver1986}.
Let $c_H$ be the corresponding objective vector: its entries are $1$ for
columns from $C$ and $0$ for columns from $-\mathrm{Id}_n$. Then
$$
\chif(G)=c_H^{\mathsf T}H^{-1}\mathbf 1_n,
$$
so Cramer's rule gives $q(G)\mid|\det H|$. Replace each selected column of
$-\mathrm{Id}_n$ by its negative. The resulting binary matrix $A$ has
$|\det A|=|\det H|$, and every column of $A$ is the indicator of an
independent set of $G$.
\end{proof}

Let $a_1,\ldots,a_n$ be the columns of $A$, put $a_0=\mathbf 0_n$ and
$N=n+1$, and choose $X$ uniformly from $a_0,\ldots,a_n$. Write
$\mu=\mathbb E X$ and
\begin{equation}\label{eq:upper-covariance}
\Sigma=\operatorname{Cov}(X)
:=\mathbb E[(X-\mu)(X-\mu)^{\mathsf T}]
=\mathbb E[XX^{\mathsf T}]-\mu\mu^{\mathsf T}.
\end{equation}
The covariance matrix $\Sigma$ is positive semidefinite, since
$z^{\mathsf T}\Sigma z=\mathbb E[(z^{\mathsf T}(X-\mu))^2]\ge0$ for every $z\in\mathbb R^n$. Each of the
$n+1$ columns is sampled with probability $1/(n+1)$. Writing the
bordered matrix and its Gram matrix gives
\begin{equation}\label{eq:upper-bordered}
\begin{aligned}
Y&=\begin{pmatrix}1&\mathbf 1_n^{\mathsf T}\\\mathbf 0_n&A\end{pmatrix},
\\
YY^{\mathsf T}
&=\sum_{c=0}^n
\begin{pmatrix}1\\a_c\end{pmatrix}
\begin{pmatrix}1&a_c^{\mathsf T}\end{pmatrix}
=N\begin{pmatrix}1&\mu^{\mathsf T}\\\mu&\mathbb E[XX^{\mathsf T}]\end{pmatrix}.
\end{aligned}
\end{equation}
Taking the Schur complement of the top-left entry in the last matrix,
and using $\det Y=\det A$, gives
\begin{equation}\label{eq:upper-det-covariance}
|\det A|^2=\det(YY^{\mathsf T})=N^{n+1}\det\Sigma.
\end{equation}

For $B\subseteq V(G)$, let $X_B$ be the subvector indexed by $B$.
Given a partition $V(G)=B_1\cup\cdots\cup B_t$, the Hadamard--Fischer
inequality applied to the principal blocks
$\operatorname{Cov}(X_{B_i})$ gives
\begin{equation}\label{eq:upper-hadamard-fischer}
\det\Sigma\le\prod_{i=1}^t\det\operatorname{Cov}(X_{B_i});
\end{equation}
see~\cite[Chapter~7]{HornJohnson2013}.
Taking singleton blocks and using $\operatorname{Var}(X_v)\le1/4$
recovers $|\det A|\le D(n)$. The following lemma uses the edges of a
star to improve the estimate on a larger block.

For an integer $m\ge2$, define
\begin{equation}\label{eq:upper-star-factor}
f(m)=2\sqrt{\frac{m^m}{(m+1)^{m+1}}},
\qquad \gamma=f(2)=\frac4{3\sqrt3}.
\end{equation}

\begin{lemma}\label{lem:upper-star-covariance}
Let $X_*=(J,T^{\mathsf T})^{\mathsf T}\in\{0,1\}^{m\times1}$ be the
indicator of a random independent set of $K_{1,m-1}$, where $J$ is the
centre coordinate and $T\in\{0,1\}^{(m-1)\times1}$ is the leaf vector.
Then
\begin{equation}\label{eq:upper-star-cov-bound}
\det\operatorname{Cov}(X_*)\le4^{-m}f(m)^2.
\end{equation}
The bound is attained by a suitable distribution.
\end{lemma}

\begin{proof}
Put $\Sigma_*=\operatorname{Cov}(X_*)$ and $p=\Pr(J=1)$.
If $p\in\{0,1\}$, then the centre row and column of $\Sigma_*$ are zero,
so the bound holds. Suppose $0<p<1$, and define
$$
\alpha=\mathbb E[T\mid J=0],
\qquad Q=\operatorname{Cov}(T\mid J=0).
$$
Since $T=\mathbf 0_{m-1}$ when $J=1$, conditioning on $J$ gives
$$
\mathbb E T=(1-p)\alpha,
\qquad \mathbb E[JT^{\mathsf T}]=\mathbf 0_{m-1}^{\mathsf T},
\qquad \mathbb E[TT^{\mathsf T}]=(1-p)(Q+\alpha\alpha^{\mathsf T}).
$$
Thus
\begin{equation}\label{eq:upper-star-cov-matrix}
\Sigma_*=
\begin{pmatrix}
p(1-p)&-p(1-p)\alpha^{\mathsf T}\\
-p(1-p)\alpha&(1-p)Q+p(1-p)\alpha\alpha^{\mathsf T}
\end{pmatrix}.
\end{equation}
Adding $\alpha_i$ times the first row to row $i+1$ eliminates the
lower-left block and replaces the lower-right block by $(1-p)Q$.
Consequently,
\begin{equation}\label{eq:upper-star-det}
\det\Sigma_*=p(1-p)^m\det Q.
\end{equation}
Since $Q$ is a covariance matrix with diagonal entries at most $1/4$,
Hadamard's inequality gives $\det Q\le4^{-(m-1)}$.
Combining this with~\eqref{eq:upper-star-det} and
$$
\max_{0\le p\le1}p(1-p)^m=\frac{m^m}{(m+1)^{m+1}}
$$
proves~\eqref{eq:upper-star-cov-bound}.
Equality holds when $p=1/(m+1)$ and, conditional on $J=0$, the leaf
coordinates are independent with mean $1/2$.
\end{proof}

For $m=2$, the covariance determinant is at most $1/27$, compared with
$1/16$ when the two binary coordinates are unrestricted. Thus each
disjoint edge contributes the factor $\gamma$ to the determinant bound.

\begin{proposition}\label{prop:upper-star-forest-bound}
If $G$ has a spanning star forest with component orders
$m_1,\ldots,m_t\ge2$, then
\begin{equation}\label{eq:upper-star-product}
q(G)\le D(n)\prod_{i=1}^t f(m_i).
\end{equation}
If $G$ has a matching of size $s$, then $q(G)\le D(n)\gamma^s$.
\end{proposition}

\begin{proof}
Partition the covariance matrix into the star vertex sets and apply
\eqref{eq:upper-hadamard-fischer}, Lemma~\ref{lem:upper-star-covariance},
\eqref{eq:upper-det-covariance} and Lemma~\ref{lem:upper-basis}.
The stars need not be induced and for a matching, use its edges as blocks
and the unmatched vertices as singleton blocks.
\end{proof}

For the star forest in Proposition~\ref{prop:upper-star-forest-bound}, an
upper bound on the component orders will suffice. Indeed, if
$2\le m_i\le R$ and $\sum_i m_i=n$, then
\begin{equation}\label{eq:upper-bounded-orders}
\prod_i f(m_i)\le\left(1+\frac R4\right)^{-n/(2R)}.
\end{equation}
The binomial theorem gives $(1+1/m)^m\ge2$ for integers $m\ge2$, thus
$$
f(m)^2\le\frac2{m+1}\le\frac4{m+4}
=\left(1+\frac m4\right)^{-1}.
$$
For $m\le R$, the weighted arithmetic--geometric mean inequality gives
$$
1+\frac m4
=\left(1-\frac mR\right)\cdot1
+\frac mR\left(1+\frac R4\right)
\ge\left(1+\frac R4\right)^{m/R}.
$$
Multiplying these estimates proves~\eqref{eq:upper-bounded-orders}.

\subsection{Small maximum matchings and completion of the proof}
\label{subsec:upper-small-matching}

Fix a maximum matching $M=\{u_iv_i:1\le i\le s\}$ and let $U$ be the
unmatched vertices. Define
$$
W=\{w\in U:\{u_i,v_i\}\subseteq N_G(w)\text{ for some }i\},
\qquad U_0=U\setminus W,
$$
and put $\eta=|W|$ and $a=|U_0|$, so $n=2s+\eta+a$.
Maximality of $M$ implies that $U$ is independent. Since $M$ is maximum,
two distinct vertices of $U$ cannot be adjacent to opposite ends of the
same matching edge: replacing that edge by the two incident edges would
enlarge $M$. Thus a matching edge can witness membership in $W$ for at
most one vertex, and $\eta\le s$.

No vertex of $U_0$ is adjacent to both ends of a matching edge. Together
with the preceding observation, this shows that all edges from $U_0$ to
$\{u_i,v_i\}$ meet one common endpoint. Denote it by $z_i$ and its partner
by $r_i$; if there is no such edge, choose either endpoint as $z_i$.
Put $Z=\{z_1,\ldots,z_s\}$. The neighbourhoods of vertices in $U_0$ are
nonempty subsets of $Z$ and form an antichain
by~\eqref{eq:upper-incomparable}. For integers $t\ge0$, put
$c_t=\binom{t}{\lfloor t/2\rfloor}$.

\begin{lemma}\label{lem:upper-bounded-star-forest}
If an integer $h\ge1$ satisfies $c_h<a$, then $G$ has a spanning star
forest with $s$ components of orders $m_1,\ldots,m_s$ satisfying
\begin{equation}\label{eq:upper-star-orders}
2\le m_i\le\left\lceil\frac ah\right\rceil+3\le\frac nh+4,
\qquad \sum_{i=1}^s m_i=n.
\end{equation}
\end{lemma}

\begin{proof}
Put $b=\lceil a/h\rceil$. For a nonempty $S\subseteq U_0$, let
$$
N_Z(S)=\bigcup_{v\in S}N_G(v),
\qquad t=|N_Z(S)|\ge1.
$$
The neighbourhoods of vertices in $S$ are distinct incomparable subsets
of this $t$-element set, so Sperner's theorem~\cite{Sperner1928} gives
$|S|\le c_t$. Comparing consecutive central binomial coefficients shows
that $c_t/t$ is nondecreasing for $t\ge1$. Thus, if $t\le h$, then
$$
\frac{|S|}{t}\le\frac{c_t}{t}\le\frac{c_h}{h}<\frac ah\le b;
$$
if $t>h$, then $|S|/t\le a/t\le b$. Hence
\begin{equation}\label{eq:upper-hall-capacity}
|S|\le b|N_Z(S)|\qquad(S\subseteq U_0).
\end{equation}
Apply Hall's theorem~\cite{Hall1935} to the bipartite graph between $U_0$
and $Z$, replacing each vertex of $Z$ by $b$ copies.
Condition~\eqref{eq:upper-hall-capacity} gives a matching saturating $U_0$.
Projecting the copies back to $Z$ assigns each vertex of $U_0$ to a
neighbouring centre, with at most $b$ vertices assigned to each centre.

Start with the edges $z_ir_i$ and attach these assigned vertices.
For each $w\in W$, choose a witness edge and attach $w$ to its centre.
No matching edge witnesses two vertices of $W$, so each star has at most
two matched vertices, $b$ vertices of $U_0$, and one vertex of $W$.
Retaining only these edges gives a spanning star forest with component
orders at most $b+3$.
\end{proof}

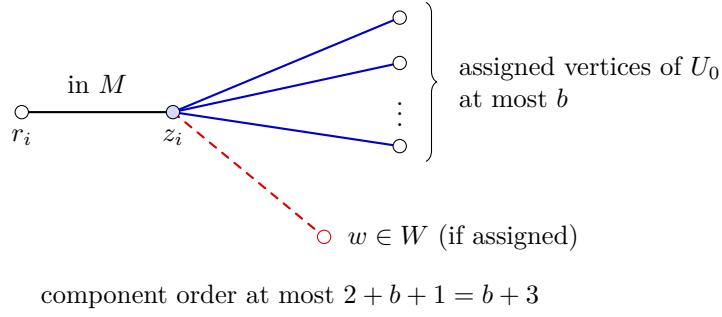
\begin{figure}[tbp]
\centering

\begin{tikzpicture}[
  x=1cm,y=1cm,
  font=\small,
  vertex/.style={circle,draw,fill=white,inner sep=0pt,minimum size=5pt},
  every path/.style={line cap=round,line join=round}
]
  \coordinate (r) at (0,0);
  \coordinate (z) at (2,0);
  \coordinate (u1) at (5,1.25);
  \coordinate (u2) at (5,0.65);
  \coordinate (u3) at (5,-0.45);
  \coordinate (w) at (4,-1.65);

  \draw[thick] (r) -- node[above=3pt] {in $M$} (z);
  \draw[blue!80!black,thick] (z) -- (u1);
  \draw[blue!80!black,thick] (z) -- (u2);
  \draw[blue!80!black,thick] (z) -- (u3);
  \draw[red!85!black,thick,dashed] (z) -- (w);

  \node[vertex,label=below:$r_i$] at (r) {};
  \node[vertex,fill=blue!15,label=below:$z_i$] at (z) {};
  \node[vertex] at (u1) {};
  \node[vertex] at (u2) {};
  \node at (5,0.08) {$\vdots$};
  \node[vertex] at (u3) {};
  \node[vertex,draw=red!85!black] at (w) {};
  \node[anchor=west] at (4.18,-1.65) {$w\in W$ (if assigned)};

  \draw[decorate,decoration={brace,amplitude=4pt}]
    (5.35,1.45) -- (5.35,-0.65);
  \node[anchor=west,align=left] at (5.65,0.40)
    {assigned vertices of $U_0$\\at most $b$};
  \node[anchor=north] at (3.55,-2.15)
    {component order at most $2+b+1=b+3$};
\end{tikzpicture}
\caption{A star obtained from a matching edge by the Hall assignment.
Only retained edges are drawn. Each centre receives at most $b$ vertices
of $U_0$ and at most one vertex of $W$.}
\label{fig:upper-star-forest}
\end{figure}

\begin{proof}[Proof of Theorem~\ref{thm:upper-main}]
Put $L=\log_2 n$, and let $s$ be the size of a maximum matching.
All estimates below hold for sufficiently large $n$.
If $s\ge2L^2$, Proposition~\ref{prop:upper-star-forest-bound} gives
\begin{equation}\label{eq:upper-large-matching}
\frac{q(G)}{D(n)}
\le\gamma^s
=\left(\frac{16}{27}\right)^{s/2}
\le2^{-s/4}
\le2^{-L^2/2}
=n^{-L/2},
\end{equation}
where $16/27<2^{-1/2}$.

Suppose $s<2L^2$. Then
$a=n-2s-\eta>n-6L^2>n/2$. Choose
\begin{equation}\label{eq:upper-h-choice}
h=\left\lfloor L+\frac12\log_2 L-2\right\rfloor.
\end{equation}
The elementary central-binomial estimate
$c_t\le2^t/\sqrt{t+1}$
and the fact that $h\ge L$ eventually give
$$
c_h\le\frac{2^h}{\sqrt h}
\le\frac{n\sqrt L}{4\sqrt h}
\le\frac n4<a.
$$
Lemma~\ref{lem:upper-bounded-star-forest} now gives a spanning star forest
with $m_i\le n/h+4$. Set $R=n/(h-1)$. Since $h=O(\log n)$, we have
$4h(h-1)\le n$ eventually, so $m_i\le n/h+4\le R$.
By~\eqref{eq:upper-star-product} and~\eqref{eq:upper-bounded-orders},
\begin{equation}\label{eq:upper-small-matching}
\frac{q(G)}{D(n)}
\le\left(1+\frac{n}{4(h-1)}\right)^{-(h-1)/2}
\le\left(\frac{4(h-1)}n\right)^{(h-1)/2}.
\end{equation}
As $L+\frac12\log_2 L-4\le h-1\le2L$ and $0<8L/n<1$, we have
\begin{equation}\label{eq:upper-final-estimate}
\begin{aligned}
\frac{q(G)}{D(n)}
&\le\left(\frac{8L}n\right)^{(L+\frac12\log_2L-4)/2}\\
&=2^{-\frac12(L+\frac12\log_2L-4)(L-\log_2L-3)}\\
&\le n^{-L/2+(\log_2L)/4+4}.
\end{aligned}
\end{equation}
Indeed, writing $z=\log_2 L$, the product in the exponent is
$$
L^2-\frac12Lz-7L-\frac12z^2+\frac52z+12
\ge L^2-\frac12Lz-8L
$$
for large $L$.
Together with~\eqref{eq:upper-large-matching}, this proves the theorem.
\end{proof}

\section{Proof of the lower bound}\label{sec:towers}\label{low:section}

We use the tower construction of Bonamy et al.\ to obtain large reduced
denominators. The main step is an exact evaluation of the fractional
chromatic number for every complete seed. We then show that the resulting
fraction is in lowest terms and estimate its denominator.

Write $\operatorname{Ind}(G)$ for the family of all independent sets of $G$,
including the empty set. For nonnegative weights $\lambda$ on
$\operatorname{Ind}(G)$ and $\mu$ on $V(G)$, put
$$
\lambda[v]=\sum_{\substack{I\in\operatorname{Ind}(G)\\v\in I}}\lambda(I),
\qquad
\mu(U)=\sum_{v\in U}\mu(v)\quad(U\subseteq V(G)).
$$
Thus $\lambda$ is a fractional colouring if $\lambda[v]\ge1$ for every
vertex $v$, and $\mu$ is a fractional clique if $\mu(I)\le1$ for every
$I\in\operatorname{Ind}(G)$. Let $\omega_f(G)$ be the maximum total weight
$\mu(V(G))$ of a fractional clique. Linear-programming duality gives
\begin{equation}\label{eq:duality}
\omega_f(G)=\chif(G).
\end{equation}
See Scheinerman and Ullman~\cite{ScheinermanUllman2011}.

\subsection{The tower construction}\label{subsec:tower-construction}\label{low:construction}

Let $H$ be a graph, and let $k\ge\ell\ge1$ be integers. Take the $\ell$th
power of the path on the $k+\ell$ vertices
$p_{1-\ell},\ldots,p_k$, and form its tensor product with $H$.
Denote the product vertices by $p_{i,h}$, where $1-\ell\le i\le k$ and
$h\in V(H)$. Two such vertices $p_{i,h}$ and $p_{j,h'}$ are adjacent exactly
when
$$
1\le |i-j|\le\ell\qquad\text{and}\qquad hh'\in E(H).
$$
For each $i\in\{1-\ell,\ldots,0\}$, identify all vertices of the fibre
$\{p_{i,h}:h\in V(H)\}$ into one vertex $z_{i+\ell}$, suppressing parallel
edges. Finally, add the edges of a copy of $H$ inside the terminal fibre
$\{p_{k,h}:h\in V(H)\}$. The resulting simple graph is $T_{k,\ell}(H)$.
This is the tower construction of Bonamy
et al.~\cite[Definition~12]{BonamyHylasovaKaiserSereni2026}.

Throughout this section, take $H=K_m$ with $m\ge2$. For a positive integer
$r$, write $[r]=\{1,\ldots,r\}$, and label $V(K_m)=[m]$. Put
$$
Z=\{z_1,\ldots,z_\ell\},
\qquad
P_i=\{p_{i,a}:a\in[m]\}\quad(i\in[k]).
$$
For $1\le s\le t\le k$, write $P(s,t)=\bigcup_{i=s}^{t}P_i$, and set
$P(s,t)=\varnothing$ when $s>t$. The order of the tower is
\begin{equation}\label{eq:order}
|V(T_{k,\ell}(K_m))|=mk+\ell.
\end{equation}

\begin{observation}\label{obs:adjacency}
In $T_{k,\ell}(K_m)$ the following hold.
\begin{enumerate}
\item[(i)] $Z$ induces a clique.
\item[(ii)] For $1\le t\le\ell$ and $1\le i\le k$, the vertex $z_t$ is
adjacent to every vertex of $P_i$ if and only if $i\le t$.
\item[(iii)] If $1\le i<j\le k$ and $a,b\in[m]$, then $p_{i,a}p_{j,b}$ is
an edge if and only if $j-i\le\ell$ and $a\ne b$.
\item[(iv)] The layers $P_1,\ldots,P_{k-1}$ are independent, whereas $P_k$
induces $K_m$.
\end{enumerate}
\end{observation}

\begin{proof}
Distinct fibres at path distance at most $\ell$ are joined by all pairs
with different labels. This gives (i) and (iii). Since $z_t$ comes from
the fibre at index $t-\ell$, it is adjacent to $P_i$ exactly when $i\le t$,
proving (ii). Each nonterminal fibre remains independent, and the added
terminal edges form $K_m$, proving (iv).
\end{proof}

Figure~\ref{fig:general} illustrates these adjacency rules.
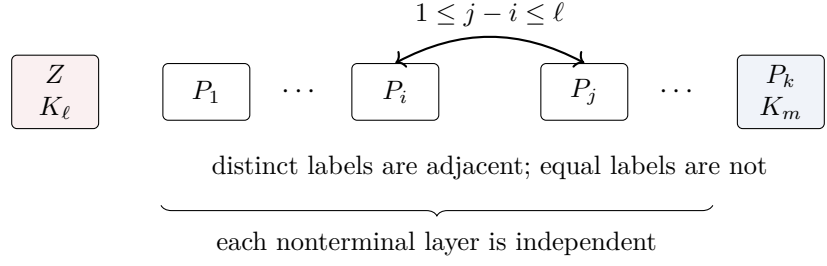
\begin{figure}[tbp]
\centering
\begin{tikzpicture}[x=1cm,y=1cm,
 box/.style={draw,rounded corners=2pt,minimum width=1.15cm,
 minimum height=.75cm,align=center,font=\small}]
 \node[box,fill=setred!7] (Z) at (0,0) {$Z$\\$K_\ell$};
 \node[box] (P1) at (2,0) {$P_1$};
 \node at (3.25,0) {$\cdots$};
 \node[box] (Pi) at (4.5,0) {$P_i$};
 \node[box] (Pj) at (7,0) {$P_j$};
 \node at (8.25,0) {$\cdots$};
 \node[box,fill=setblue!7] (Pk) at (9.6,0) {$P_k$\\$K_m$};
 \draw[<->,thick] (Pi.north) to[bend left=28]
 node[above,font=\small] {$1\le j-i\le\ell$} (Pj.north);
 \node[align=center,font=\small] at (5.75,-1.0)
 {distinct labels are adjacent; equal labels are not};
 \draw[decorate,decoration={brace,mirror,amplitude=4pt}]
 (1.4,-1.5)--(8.7,-1.5)
 node[midway,below=6pt,font=\small] {each nonterminal layer is independent};
\end{tikzpicture}
\caption{The complete-seed tower. The cap vertex $z_t$ is joined to
all vertices in $P_1,\ldots,P_t$. Layers farther than $\ell$ apart
have no edges between them. The terminal layer is the only positive
layer with internal edges.}
\label{fig:general}
\end{figure}

\subsection{The weight sequence}\label{subsec:scaled-recurrence}

We use a single sequence for both weight assignments: its terms are read
backwards for vertex weights and forwards for independent-set weights.
Fix integers $m\ge2$ and $\ell\ge1$, and define $F_i^{(m,\ell)}$ by
\begin{equation}\label{eq:F-full}
F_i^{(m,\ell)}=0\quad(i\le\ell-2),
\qquad
F_{\ell-1}^{(m,\ell)}=\frac1{m-1},
\end{equation}
and
\begin{equation}\label{eq:F-rec-full}
(m-1)F_i^{(m,\ell)}=
\sum_{j=i-\ell}^{i-1}F_j^{(m,\ell)}\qquad(i\ge\ell).
\end{equation}
For integers $a\ge0$ and $b\ge1$, let
\begin{equation}\label{eq:S-full}
S_{m,\ell}(a,b)=\sum_{j=b-1}^{a+b-2}F_j^{(m,\ell)},
\end{equation}
with an empty sum equal to zero. Thus the first argument is the number of
terms, and the second specifies the starting index $b-1$. When $m$ and
$\ell$ are fixed, we omit them from the notation and write $F_i$ and
$S(a,b)$. The recurrence gives $F_i\ge0$ for all $i$ and $F_i>0$ for
$i\ge\ell-1$.

Concatenating two consecutive intervals of indices gives
\begin{equation}\label{eq:concat}
S(c,a+b)+S(a,b)=S(a+c,b)
\qquad(a,c\ge0,\ b\ge1).
\end{equation}
We will also use the two forms
\begin{align}
(m-1)F_{q+\ell-1}&=S(\ell,q)
&& (q\ge1),\label{eq:rS}\\
mF_i&=F_{i-\ell}+F_{i-\ell+1}+\cdots+F_i
&& (i\ge\ell)\label{eq:mblock}
\end{align}
of~\eqref{eq:F-rec-full}. The first is used in the tail estimate and
the second in the colouring construction.

\subsection{The exact fractional chromatic number}\label{sec:certificates}

\begin{theorem}\label{thm:exact}
For all integers $m\ge2$ and $k\ge\ell\ge1$,
$$
\chif\bigl(T_{k,\ell}(K_m)\bigr)
=\ell+1+\frac1{S_{m,\ell}(k,\ell)}.
$$
\end{theorem}

The case $m=2$ is the tower formula of Bonamy
et al.~\cite[Theorem~5]{BonamyHylasovaKaiserSereni2026}. For $\ell=1$,
reversing the layer order gives the generalized cone over $K_m$, so the evaluation also
follows from Tardif's formula~\cite{Tardif2001}.
Theorem~\ref{thm:exact} extends the evaluation to $m\ge3$ and $\ell\ge2$.

Fix $k\ge\ell\ge1$, write $G=T_{k,\ell}(K_m)$, and put
$S_*=S_{m,\ell}(k,\ell)$. We construct a fractional clique and a
fractional colouring of the same weight $\ell+1+1/S_*$. These give the
lower and upper bounds, respectively. The constructions extend the two
certificates in~\cite[Sections~3.1--3.2]{BonamyHylasovaKaiserSereni2026}:
the tail estimate extends their Lemma~8 to $m$ labels, and the
independent sets used for the colouring extend those in their
Lemmas~9 and~10.

\subsubsection{The fractional-clique construction}

Define nonnegative vertex weights $\mu$ by
\begin{align}
\mu(p_{i,a})&=F_{k-i+\ell-1}
&& (1\le i\le k,\ a\in[m]),\label{eq:mu-p}\\
\mu(z_j)&=S(j,k+\ell-j)
&& (1\le j\le\ell).\label{eq:mu-x}
\end{align}
The completeness of the seed gives two useful restrictions. An
independent set meeting a nonterminal layer in at least two vertices
misses each of the next $\ell$ layers that exists. It also contains at
most one vertex of $P_k$. These restrictions are used in the following
tail estimate.

\begin{lemma}\label{lem:tail}
Let $j\in\{1,\ldots,k+1\}$. If $J$ is an independent set contained in
$P(j,k)$, then
$$
\mu(J)\le S(k-j+1,\ell).
$$
\end{lemma}

\begin{proof}
We use induction on $k-j+1$. The case $j=k+1$ is immediate. Suppose
$j\le k$, and put $J'=J\setminus P_j$ and $c=|J\cap P_j|$. Then
$$
\mu(J)=\mu(J')+cF_{k-j+\ell-1}.
$$
If $c\le1$, the induction hypothesis applied to $J'\subseteq P(j+1,k)$
gives
$$
\mu(J)\le S(k-j,\ell)+F_{k-j+\ell-1}=S(k-j+1,\ell).
$$

Suppose $c\ge2$. Since $P_k$ is a clique, $j<k$. Every vertex in the
next $\ell$ layers that exist has a label different from that of at least
one of two chosen vertices in $J\cap P_j$, and is therefore adjacent to
it by Observation~\ref{obs:adjacency}(iii). Hence
$$
J'\subseteq P(j+\ell+1,k).
$$
Put $q=k-j\ge1$. If $q\ge\ell$, the induction hypothesis gives
$\mu(J')\le S(q-\ell,\ell)$, including the empty-tail case $q=\ell$.
Since $c\le m$, equations~\eqref{eq:rS} and
\eqref{eq:concat} yield
\begin{align*}
\mu(J)&\le S(q-\ell,\ell)+mF_{q+\ell-1}\\
&=S(q-\ell,\ell)+S(\ell,q)+F_{q+\ell-1}\\
&=S(q,\ell)+F_{q+\ell-1}=S(q+1,\ell).
\end{align*}
If $1\le q<\ell$, all remaining layers are excluded, so $J'=\varnothing$.
The initial zero terms give
$$
S(\ell,q)=\sum_{i=q-1}^{q+\ell-2}F_i
=\sum_{i=\ell-1}^{q+\ell-2}F_i=S(q,\ell).
$$
Thus
$$
\mu(J)\le mF_{q+\ell-1}
=S(q,\ell)+F_{q+\ell-1}=S(q+1,\ell),
$$
completing the induction.
\end{proof}

Figure~\ref{fig:tail-blocking} illustrates the second case of the proof.
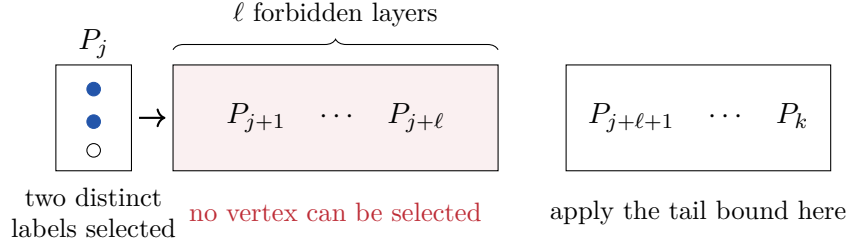
\begin{figure}[tbp]
\centering
\begin{tikzpicture}[x=1cm,y=1cm,
 block/.style={draw,minimum width=1cm,minimum height=1.4cm},
 point/.style={circle,draw,fill=white,inner sep=1.8pt}]
 \node[block] (first) at (0,0) {};
 \node[point,fill=setblue,draw=setblue] at (0,.38) {};
 \node[point,fill=setblue,draw=setblue] at (0,-.05) {};
 \node[point] at (0,-.43) {};
 \node at (0,.98) {$P_j$};
 \node[font=\small,align=center] at (0,-1.25) {two distinct\\labels selected};
 \node[draw,fill=setred!7,minimum width=4.3cm,minimum height=1.4cm]
 (blocked) at (3.2,0) {$P_{j+1}\quad\cdots\quad P_{j+\ell}$};
 \node[font=\small,text=setred,align=center] at (3.2,-1.25)
 {no vertex can be selected};
 \draw[->,thick] (.6,0)--(.95,0);
 \node[draw,minimum width=3.5cm,minimum height=1.4cm] at (8,0)
 {$P_{j+\ell+1}\quad\cdots\quad P_k$};
 \node[font=\small] at (8,-1.25) {apply the tail bound here};
 \draw[decorate,decoration={brace,amplitude=4pt}]
 (1.05,.9)--(5.35,.9) node[midway,above=5pt,font=\small] {$\ell$ forbidden layers};
\end{tikzpicture}
\caption{The second case of the tail estimate. Two different labels
in $P_j$ together exclude every vertex in each of the next $\ell$
layers that exists. If $j+\ell\ge k$, no tail remains.}
\label{fig:tail-blocking}
\end{figure}

\begin{lemma}\label{lem:mu-feasible}
Every independent set $I$ of $G$ satisfies $\mu(I)\le S_*$. Consequently,
$\mu/S_*$ is a fractional clique.
\end{lemma}

\begin{proof}
Since $Z$ is a clique, $I$ contains at most one vertex of $Z$. If
$I\cap Z=\varnothing$, Lemma~\ref{lem:tail} with $j=1$ gives the bound.
Otherwise, write $I\cap Z=\{z_t\}$. By
Observation~\ref{obs:adjacency}(ii), $I\setminus Z\subseteq P(t+1,k)$,
and hence
$$
\mu(I)\le S(k-t,\ell)+S(t,k+\ell-t)=S(k,\ell)=S_*,
$$
where the equality follows from~\eqref{eq:concat}.
\end{proof}

The total vertex weight is given by the next proposition;
compare~\cite[Assertion~7]{BonamyHylasovaKaiserSereni2026}.

\begin{proposition}\label{prop:mu-total}
The vertex weights satisfy $\mu(V(G))=(\ell+1)S_*+1$.
\end{proposition}

\begin{proof}
Each of the $m$ labels contributes $S_*$ on the positive layers, so
$$
\mu(V(G))=mS_*+\sum_{t=1}^{\ell}S(t,k+\ell-t).
$$
By~\eqref{eq:concat},
$S(t,k+\ell-t)=S_*-S(k-t,\ell)$. Also,
\begin{align*}
\sum_{t=1}^{\ell}S(k-t,\ell)
&=\sum_{i=\ell}^{k+\ell-2}\sum_{t=1}^{\ell}F_{i-t}\\
&=(m-1)\sum_{i=\ell}^{k+\ell-2}F_i
=(m-1)S_*-1.
\end{align*}
The first equality uses the zero terms below index $\ell-1$, the second
uses~\eqref{eq:F-rec-full}, and the last uses
$(m-1)F_{\ell-1}=1$. Substitution gives
$$
\mu(V(G))=mS_*+\ell S_*-((m-1)S_*-1)=(\ell+1)S_*+1.
$$
\end{proof}

By Lemma~\ref{lem:mu-feasible} and
Proposition~\ref{prop:mu-total}, $\mu/S_*$ is a fractional clique of
weight $\ell+1+1/S_*$. Thus~\eqref{eq:duality} gives the lower bound
\begin{equation}\label{eq:exact-lower}
\chif(G)\ge\ell+1+\frac1{S_*}.
\end{equation}

\subsubsection{The fractional-colouring construction}\label{sec:primal}

Set $\Delta=\ell+1$. For any integer $j$, let $\rho(j)$ be the unique
representative of $j$ modulo $\Delta$ in $\{1,\ldots,\Delta\}$, and let
$e(j)$ be the largest integer at most $k$ congruent to $j$ modulo
$\Delta$. Thus
\begin{equation}\label{eq:double}
e(j)=j+\left\lfloor\frac{k-j}{\Delta}\right\rfloor\Delta
\in\{k-\ell,\ldots,k\}.
\end{equation}
In particular, $i\le e(i)\le k$ for $i\in[k]$. For every integer $j$, put
$$
Z_j=
\begin{cases}
\{z_{\rho(j)-1}\},&\rho(j)\ne1,\\
\varnothing,&\rho(j)=1.
\end{cases}
$$
For $j\in[k]$ and $a\in[m]$, define
\begin{equation}\label{eq:A}
A_{j,a}=\{p_{i,a}:j\le i\le k\}
\cup\bigcup_{\substack{1\le i<j\\i\equiv j\pmod\Delta}}P_i
\cup Z_j.
\end{equation}
For $j\in\{k-\ell,\ldots,k-1\}$, define
\begin{equation}\label{eq:B}
B_j=\bigcup_{\substack{1\le i\le j\\i\equiv j\pmod\Delta}}P_i
\cup Z_j,
\end{equation}
and set $B_k=\varnothing$. When $k=\ell$, the index $j=0$ occurs and
$B_0=\{z_\ell\}$.

The two families use the same residue-class pattern before the switching
point. The set $A_{j,a}$ switches at $P_j$ to a one-label tail, meeting
$P_k$ only in $p_{k,a}$. For $j<k$, the set $B_j$ has no tail: it selects
the full layers in its residue class up to $j$, together with $Z_j$.
The separate convention $B_k=\varnothing$ makes the coverage formulas
uniform.

All sets in~\eqref{eq:A}--\eqref{eq:B} are independent.
Indeed, their full layers are nonterminal and are separated by multiples
of $\ell+1$. In an $A$-set, these layers are also separated from the
same-label tail by at least $\ell+1$; the tail meets the terminal clique
in only one vertex. A selected cap vertex is $z_{\rho(j)-1}$, whereas
every selected positive layer has index at least $\rho(j)$, so no cap
edge occurs. The exceptional sets $B_k=\varnothing$ and
$B_0=\{z_\ell\}$, when present, are independent as well.

Figure~\ref{fig:supports} illustrates the two types of independent sets.
\begin{figure}[tbp]
\centering
\begingroup
\newcommand{\drawtowerpanel}[2]{
\begin{tikzpicture}[x=1.12cm,y=.65cm,
 vertex/.style={circle,draw=black!45,fill=white,inner sep=1.7pt},
 chosen/.style={circle,draw=#1,fill=#1,inner sep=2.5pt},
 edge/.style={draw=black!14,line width=.35pt}]
 \foreach \i in {1,...,6}{
   \foreach \a/\yy in {1/2,2/1,3/0}{
     \coordinate (p\i\a) at ({1.65+1.55*(\i-1)},\yy);
   }
   \node[font=\small] at ({1.65+1.55*(\i-1)},2.65) {$P_{\i}$};
 }
 \coordinate (z1) at (0,1.6);
 \coordinate (z2) at (0,.4);
 \draw[edge] (z1)--(z2);
 \foreach \a in {1,2,3}{
   \draw[edge] (z1)--(p1\a);
   \draw[edge] (z2)--(p1\a);
   \draw[edge] (z2) to[bend right=8] (p2\a);
 }
 \foreach \i in {1,...,5}{
   \foreach \j in {2,...,6}{
     \pgfmathtruncatemacro{\gap}{\j-\i}
     \ifnum\gap>0\ifnum\gap<3
       \foreach \a in {1,2,3}{\foreach \b in {1,2,3}{
         \ifnum\a=\b\else
           \pgfmathtruncatemacro{\labelgap}{abs(\a-\b)}
           \ifnum\gap=2
             \ifnum\labelgap=2
               \draw[edge] (p\i\a) to[bend left=8] (p\j\b);
             \else
               \draw[edge] (p\i\a)--(p\j\b);
             \fi
           \else
             \draw[edge] (p\i\a)--(p\j\b);
           \fi
         \fi
       }}
     \fi\fi
   }
 }
 \draw[draw=black!40] (p61)--(p62)--(p63);
 \draw[draw=black!40] (p61) to[bend left=32] (p63);
 \foreach \i in {1,...,6}{\foreach \a in {1,2,3}{
   \node[vertex] at (p\i\a) {};
 }}
 \node[vertex] at (z1) {};
 \node[vertex] at (z2) {};
 \foreach \a/\yy in {1/2,2/1,3/0}{
   \node[font=\scriptsize,anchor=west] at (10,\yy) {label $\a$};
 }
 \node[font=\small,anchor=east] at (-.15,1.6) {$z_1$};
 \node[font=\small,anchor=east] at (-.15,.4) {$z_2$};
 \foreach \selected in {#2}{\node[chosen] at (\selected) {};}
\end{tikzpicture}}
{\small\color{setblue}$A_{4,2}=P_1\cup\{p_{4,2},p_{5,2},p_{6,2}\}$}\par\smallskip
\drawtowerpanel{setblue}{p11,p12,p13,p42,p52,p62}
\par\medskip
{\small\color{setred}$B_5=\{z_1\}\cup P_2\cup P_5$}\par\smallskip
\drawtowerpanel{setred}{z1,p21,p22,p23,p51,p52,p53}
\endgroup
\caption{The independent sets $A_{4,2}$ (blue) and $B_5$ (red) in
$T_{6,2}(K_3)$. Filled vertices are selected; hollow vertices are not.
All graph edges are shown in grey. The $A$-set changes from a full
layer to a single-label tail, whereas the $B$-set uses full layers
three positions apart and avoids the terminal clique.}
\label{fig:supports}
\end{figure}
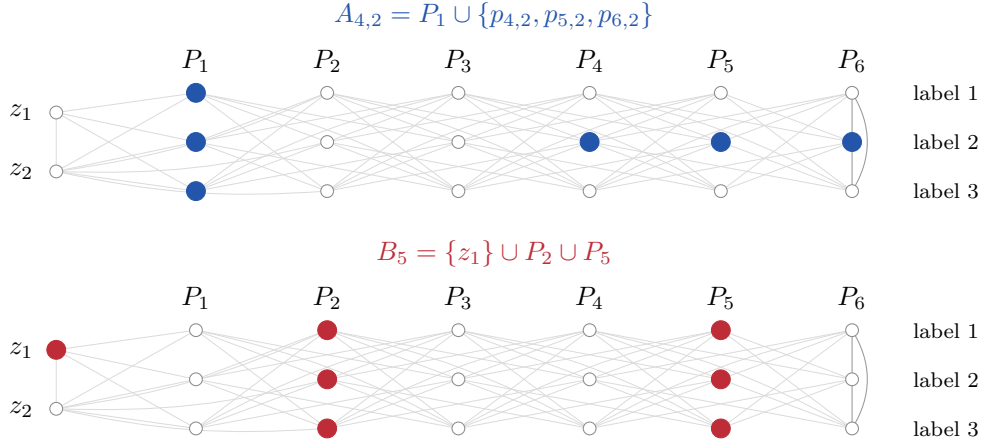

The indexed sets are pairwise distinct, so the weights below are well
defined. Different labels distinguish $A$-sets in $P_k$. For fixed $a$
and $i<j$, the set $A_{i,a}$ meets $P_i$ in one vertex, whereas $A_{j,a}$
meets it either in all $m$ vertices or not at all. Every $A$-set meets
$P_k$, and every $B$-set avoids it. Nonempty $B$-sets with positive
index are distinguished by their last full layer;
$B_0=\{z_\ell\}$, when present, and $B_k=\varnothing$ are distinct
from these and from each other.

Define nonnegative weights $\lambda$ on $\operatorname{Ind}(G)$ by
\begin{align}
\lambda(A_{j,a})&=F_{j+\ell-2}
&& (j\in[k],\ a\in[m]),\label{eq:lambdaA}\\
\lambda(B_j)&=\sum_{t=j+1}^{k}F_{t+\ell-2}
&& (k-\ell\le j\le k),\label{eq:lambdaB}
\end{align}
and set $\lambda(I)=0$ for all other independent sets. In particular,
$\lambda(B_k)=0$.

To calculate the coverage, we use~\eqref{eq:mblock} in the form
\begin{equation}\label{eq:block}
mF_{j+\ell-2}=\sum_{q=j-2}^{j+\ell-2}F_q\qquad(j\ge2).
\end{equation}
When successive indices $j$ differ by $\ell+1$, the intervals
$[j-2,j+\ell-2]$ are consecutive, with neither gaps nor overlaps.

\begin{lemma}\label{lem:coverage}
Every vertex $v\in V(G)$ satisfies $\lambda[v]=S_*$.
\end{lemma}

\begin{proof}
For $2\le i\le\Delta$, the indexed sets containing $z_{i-1}$ are the
$A_{j,a}$ with $j\equiv i\pmod\Delta$, together with $B_{e(i)}$ when
$e(i)<k$. Hence
\begin{equation}\label{eq:xcov}
\lambda[z_{i-1}]
=m\sum_{\substack{1\le j\le k\\j\equiv i\pmod\Delta}}F_{j+\ell-2}
+\sum_{t=e(i)+1}^{k}F_{t+\ell-2}.
\end{equation}
When the first sum is nonempty, its $j$-values are
$i,i+\Delta,\ldots,e(i)$. Equation~\eqref{eq:block}
expands it into the terms with indices $i-2,\ldots,e(i)+\ell-2$.
The second sum continues from the next index to $k+\ell-2$. Their sum
is $S_*$, since $i-2\le\ell-1$ and the earlier terms vanish.
The only empty-first-sum case is $k=\ell$ and $i=\ell+1$; then
$e(i)=0$, and the second sum alone equals $S_*$.

Now consider $p_{i,a}$, where $i\in[k]$ and $a\in[m]$. It lies in
$A_{j,a}$ for $1\le j\le i$, and in all $m$ sets $A_{j,b}$ when
$i<j\le k$ and $j\equiv i\pmod\Delta$. Its only possible
$B$-contribution is that of $B_{e(i)}$, interpreted as zero when
$e(i)=k$. Therefore
\begin{equation}\label{eq:pcov}
\lambda[p_{i,a}]
=\sum_{j=1}^{i}F_{j+\ell-2}
+m\sum_{\substack{i<j\le k\\j\equiv i\pmod\Delta}}F_{j+\ell-2}
+\sum_{t=e(i)+1}^{k}F_{t+\ell-2}.
\end{equation}
The first sum covers the indices $[\ell-1,i+\ell-2]$. By
\eqref{eq:block}, the second covers
$[i+\ell-1,e(i)+\ell-2]$ when nonempty, and the third covers
$[e(i)+\ell-1,k+\ell-2]$. Omitting empty intervals, these are
consecutive and give
$$
\lambda[p_{i,a}]=\sum_{q=\ell-1}^{k+\ell-2}F_q=S_*.
$$
\end{proof}

Figure~\ref{fig:coverage} summarizes the interval concatenation in
\eqref{eq:pcov}.

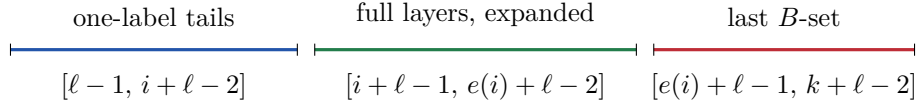
\begin{figure}[tbp]
\centering
\begin{tikzpicture}[x=1.15cm,y=1cm,font=\small]
 \draw[very thick,setblue] (0,0)--(3.3,0);
 \draw[very thick,forestgreen] (3.5,0)--(7.2,0);
 \draw[very thick,setred] (7.4,0)--(10.4,0);
 \foreach \x in {0,3.3,3.5,7.2,7.4,10.4}{\draw (\x,-.08)--(\x,.08);}
 \node[above=5pt] at (1.65,0) {one-label tails};
 \node[above=5pt] at (5.35,0) {full layers, expanded};
 \node[above=5pt] at (8.9,0) {last $B$-set};
 \node[below=5pt,align=center] at (1.65,0)
 {$[\ell-1,\,i+\ell-2]$};
 \node[below=5pt,align=center] at (5.35,0)
 {$[i+\ell-1,\,e(i)+\ell-2]$};
 \node[below=5pt,align=center] at (8.9,0)
 {$[e(i)+\ell-1,\,k+\ell-2]$};
\end{tikzpicture}
\caption{The three sums covering $p_{i,a}$ in \eqref{eq:pcov}.
The intervals refer to subscripts of $F$, not to layer indices.
After the middle sum is expanded, they concatenate to all terms
of $S(k,\ell)$. An interval is omitted when its lower endpoint
exceeds its upper endpoint.}
\label{fig:coverage}
\end{figure}

\begin{lemma}\label{lem:lambda-total}
The total weight of $\lambda$ is
$$
\sum_{I\in\operatorname{Ind}(G)}\lambda(I)=(\ell+1)S_*+1.
$$
\end{lemma}

\begin{proof}
Since $Z$ is a clique, Lemma~\ref{lem:coverage} shows that the sets
meeting $Z$ have total weight $\ell S_*$. The remaining positive-weight
sets are $A_{j,a}$ with $j\equiv1\pmod\Delta$ and, if $e(1)<k$, the
set $B_{e(1)}$. Write $e(1)=1+h\Delta$, where $h\ge0$. Including the
zero-weight set $B_k$ when necessary, their total weight is
\begin{equation}\label{eq:noX}
m\sum_{t=0}^{h}F_{1+t\Delta+\ell-2}
+\sum_{u=e(1)+1}^{k}F_{u+\ell-2}.
\end{equation}
Separate the initial term, for which $mF_{\ell-1}=1+F_{\ell-1}$.
For $t\ge1$, equation~\eqref{eq:block} expands the remaining
terms into consecutive blocks from index $\ell$ to $e(1)+\ell-2$;
the final sum continues through $k+\ell-2$. Thus
\eqref{eq:noX} is $1+S_*$. Adding $\ell S_*$ gives
$(\ell+1)S_*+1$.
\end{proof}

\begin{proof}[Proof of Theorem~\ref{thm:exact}]
By Lemmas~\ref{lem:coverage} and~\ref{lem:lambda-total},
$\lambda/S_*$ is a fractional colouring of weight $\ell+1+1/S_*$.
This matches the lower bound~\eqref{eq:exact-lower}, proving the theorem.
\end{proof}

The same colouring also proves that the towers are vertex-critical.

\begin{lemma}\label{lem:tower-critical}
For all integers $m\ge2$ and $k\ge\ell\ge1$, the graph
$T_{k,\ell}(K_m)$ is fractionally vertex-critical.
\end{lemma}

\begin{proof}
Put $G=T_{k,\ell}(K_m)$ and $S_*=S_{m,\ell}(k,\ell)$, and let $w$ be
any optimal fractional clique of $G$. Write $w(U)=\sum_{v\in U}w(v)$.
The preceding optimal colouring $\lambda/S_*$ covers every vertex with
weight exactly $1$, so
$$
\sum_{I\in\operatorname{Ind}(G)}\frac{\lambda(I)}{S_*}\bigl(1-w(I)\bigr)=0.
$$
Every summand is nonnegative. The weights in
\eqref{eq:lambdaA}--\eqref{eq:lambdaB} are positive for all $A_{j,a}$
and all $B_j$ with $j<k$. Hence
$$
w(A_{j,a})=1\quad(j\in[k],\ a\in[m]),
\qquad
w(B_j)=1\quad(k-\ell\le j<k).
$$

Comparing the equalities for $A_{j,a}$ over the labels shows that
$\sum_{i=j}^{k}w(p_{i,a})$ is independent of $a$. Starting with $j=k$
and subtracting consecutive tail sums, we obtain
$w(p_{i,a})=x_i$ for some $x_i\ge0$ independent of $a$.
For $1\le j<k$, compare $A_{j,a}$ with $A_{j+\ell+1,a}$ when
$j+\ell+1\le k$, and with $B_j$ otherwise. By
\eqref{eq:A}--\eqref{eq:B}, the common vertices cancel, giving
\begin{equation}\label{eq:critical-recurrence}
(m-1)x_j=\sum_{i=j+1}^{\min\{j+\ell,k\}}x_i
\qquad(1\le j<k).
\end{equation}
Also, $w(A_{1,a})=1$ gives $\sum_{i=1}^{k}x_i=1$.
If $x_k=0$, backward substitution in~\eqref{eq:critical-recurrence}
would force every $x_i$ to be zero. Thus $x_k>0$, and the same
recurrence gives $x_i>0$ for every $i$.

For $1\le t\le\ell$ with $t<k$, the set $A_{t+1,a}$ consists of
$z_t$ and the one-label tail from $P_{t+1}$ to $P_k$. Therefore
$$
w(z_t)=1-\sum_{i=t+1}^{k}x_i=\sum_{i=1}^{t}x_i>0.
$$
The only remaining case is $t=k=\ell$, when the tight set
$B_0=\{z_\ell\}$ gives $w(z_\ell)=1$. Thus every optimal fractional
clique is positive at every vertex. If $\chif(G-v)=\chif(G)$ for some
$v$, an optimal fractional clique of $G-v$, extended by weight zero at
$v$, would be optimal for $G$. This contradiction proves the lemma.
\end{proof}

\subsection{The reduced denominator}\label{sec:denominator}

We now determine the reduced denominator in
Theorem~\ref{thm:exact}.

\begin{theorem}\label{thm:denominator}\label{cor:den}
Let $m\ge2$ and $\ell\ge1$ be integers, and define
$$
Q_k=Q_k^{(m,\ell)}
=(m-1)^kS_{m,\ell}(k,\ell)\qquad(k\ge1).
$$
For $k\ge\ell$, the reduced denominator of
$\chif(T_{k,\ell}(K_m))$ is $Q_k$, and $Q_k$ satisfies 
\begin{equation}\label{eq:Qrec}
    Q_k=mQ_{k-1}-(m-1)^\ell Q_{k-\ell-1}
\quad(k\ge2)
\end{equation}
\end{theorem}

\begin{proof}
Write $S_k=S_{m,\ell}(k,\ell)$ for $k\ge1$, and set
$S_t=0$ for $t\le0$.

We first prove the recurrence \eqref{eq:Qrec}.
By the definition of $S_k$,
$$
S_k=S_{k-1}+F_{k+\ell-2}\qquad(k\ge2).
$$
On the other hand, \eqref{eq:F-rec-full} gives
$$
(m-1)F_{k+\ell-2}=S_{k-1}-S_{k-\ell-1}.
$$
Combining these identities, we obtain
$$
\begin{aligned}
(m-1)S_k
&=(m-1)S_{k-1}+S_{k-1}-S_{k-\ell-1}\\
&=mS_{k-1}-S_{k-\ell-1}.
\end{aligned}
$$
Multiplying both sides by $(m-1)^{k-1}$ gives
$$
Q_k=mQ_{k-1}-(m-1)^{k-1}S_{k-\ell-1}.
$$
By the definition of $Q_{k-\ell-1}$,
$$
(m-1)^{k-1}S_{k-\ell-1}
=(m-1)^\ell Q_{k-\ell-1},
$$
where both sides are zero when $k-\ell-1\le0$.
Thus
$$
Q_k=mQ_{k-1}-(m-1)^\ell Q_{k-\ell-1},
$$
as required.

The initial value is $Q_1=(m-1)F_{\ell-1}=1$.
Hence the recurrence shows inductively that every $Q_k$
is an integer and
$$
Q_k\equiv1\pmod{m-1}.
$$
In particular, $\gcd(Q_k,m-1)=1$.
Also, $Q_k>0$ since $S_k>0$.

Finally, for $k\ge\ell$, Theorem~\ref{thm:exact} gives
$$
\chif(T_{k,\ell}(K_m))
=\ell+1+\frac{1}{S_k}
=\frac{(\ell+1)Q_k+(m-1)^k}{Q_k}.
$$
Since $\gcd(Q_k,(m-1)^k)=1$, the numerator is coprime
to $Q_k$. Therefore the reduced denominator is $Q_k$.
\end{proof}

\subsection{Growth of the reduced denominators}\label{sec:asymptotic}

We next estimate the difference between $Q_k^{(m,\ell)}$
and $m^{k-1}$. The estimate holds for all $\ell\ge1$,
allowing us to choose $\ell$ as a function of $k$.

\begin{lemma}\label{lem:elementary-growth}
For all integers $m\ge2$, $\ell\ge1$ and $k\ge1$,
\begin{equation}\label{eq:Q-sandwich}
 m^{k-1}\left(1-\frac{k}{m}
       \left(\frac{m-1}{m}\right)^\ell\right)
 \le Q_k^{(m,\ell)}\le m^{k-1}.
\end{equation}
\end{lemma}

\begin{proof}
Fix $m$ and $\ell$, and write $Q_i=Q_i^{(m,\ell)}$. The recurrence~\eqref{eq:Qrec} and the nonnegativity of the $Q_i$ give $Q_k\le mQ_{k-1}$ for $k\ge2$. Since $Q_1=1$, induction gives the upper bound. For $2\le i\le k$, multiply $Q_i-mQ_{i-1}=-(m-1)^\ell Q_{i-\ell-1}$ by $m^{k-i}$ and sum. The left-hand side telescopes, giving
\begin{equation}\label{eq:growth-telescope}
 Q_k-m^{k-1}
 =-(m-1)^\ell\sum_{i=2}^{k}m^{k-i}Q_{i-\ell-1}.
\end{equation}
The terms with $i\le\ell+1$ vanish. For each remaining term, the upper bound already proved gives
$$
 m^{k-i}Q_{i-\ell-1}
 \le m^{k-i}m^{i-\ell-2}=m^{k-\ell-2}.
$$
There are at most $k$ such terms. Substitution in
\eqref{eq:growth-telescope} proves the lower bound in
\eqref{eq:Q-sandwich}.
\end{proof}

\begin{proof}[Proof of Theorem~\ref{thm:asymptotic}]
Fix an integer $m\ge2$. For each positive integer $\ell$, let
\begin{equation}\label{eq:growth-parameters}
 G_\ell=T_{\ell^2,\ell}(K_m),
 \qquad
 n_\ell=m\ell^2+\ell,
 \qquad
 q_\ell=Q_{\ell^2}^{(m,\ell)}.
\end{equation}
By Lemma~\ref{lem:tower-critical} and
Theorem~\ref{thm:denominator}, the graph $G_\ell$ is fractionally
vertex-critical, has $n_\ell$ vertices, and has reduced denominator
$q_\ell$. Hence $\sd(n_\ell)\ge q_\ell$.

Lemma~\ref{lem:elementary-growth} gives
\begin{equation}\label{eq:growth-subsequence}
 1-\frac{\ell^2}{m}
       \left(\frac{m-1}{m}\right)^\ell
 \le \frac{q_\ell}{m^{\ell^2-1}}
 \le 1.
\end{equation}
Since $m$ is fixed and $(m-1)/m<1$, the error term tends to zero.
Thus $q_\ell=(1+o(1))m^{\ell^2-1}$, and consequently
\begin{equation}\label{eq:tower-root-limit}
 q_\ell^{1/n_\ell}
 =(1+o(1))m^{(\ell^2-1)/(m\ell^2+\ell)}
 \longrightarrow m^{1/m}
 \qquad\text{as }\ell\to\infty.
\end{equation}
As $n_\ell$ is strictly increasing, it follows that
\begin{equation}\label{eq:general-exponential-base}
 \limsup_{n\to\infty}\sd(n)^{1/n}\ge m^{1/m}.
\end{equation}

Finally, 
$$
 \frac{d}{dx}\left(\frac{\log x}{x}\right)
 =\frac{1-\log x}{x^2}<0
 \qquad\text{for }x\ge3.
$$
Thus $m^{1/m}$ decreases for integers $m\ge3$.
Since $3^{1/3}>2^{1/2}$, its maximum over integers $m\ge2$
is attained at $m=3$. Substituting $m=3$ into
\eqref{eq:general-exponential-base} proves the theorem.
\end{proof}

\newpage
\bibliographystyle{plainnat}
\bibliography{references}

@article{AlonTuzaVoigt1997,
  author = {Alon, Noga and Tuza, Zsolt and Voigt, Margit},
  title = {Choosability and fractional chromatic numbers},
  journal = {Discrete Mathematics},
  volume = {165--166},
  pages = {31--38},
  year = {1997},
  doi = {10.1016/S0012-365X(96)00159-8}
}

@article{BonamyHylasovaKaiserSereni2026,
  author = {Bonamy, Marthe and Hylasov{\'a}, Karol{\'i}na and Kaiser, Tom{\'a}{\v s} and Sereni, Jean-S{\'e}bastien},
  title = {Lower bound on the maximum denominator of fractional chromatic numbers},
  journal = {The Electronic Journal of Combinatorics},
  volume = {33},
  number = {2},
  pages = {P2.38},
  year = {2026},
  doi = {10.37236/14524}
}

@article{ChvatalGareyJohnson1978,
  author = {Chv{\'a}tal, V{\'a}clav and Garey, Michael R. and Johnson, David S.},
  title = {Two results concerning multicoloring},
  journal = {Annals of Discrete Mathematics},
  volume = {2},
  pages = {151--154},
  year = {1978},
  doi = {10.1016/S0167-5060(08)70329-7}
}

@article{Fisher1995,
  author = {Fisher, David C.},
  title = {Fractional colorings with large denominators},
  journal = {Journal of Graph Theory},
  volume = {20},
  number = {4},
  pages = {403--409},
  year = {1995},
  doi = {10.1002/jgt.3190200403}
}

@article{Hall1935,
  author = {Hall, Philip},
  title = {On representatives of subsets},
  journal = {Journal of the London Mathematical Society},
  volume = {s1-10},
  number = {1},
  pages = {26--30},
  year = {1935},
  doi = {10.1112/jlms/s1-10.37.26}
}

@book{HornJohnson2013,
  author = {Horn, Roger A. and Johnson, Charles R.},
  title = {Matrix Analysis},
  publisher = {Cambridge University Press},
  edition = {2nd},
  year = {2013},
  doi = {10.1017/CBO9781139020411}
}

@article{LarsenProppUllman1995,
  author = {Larsen, Michael and Propp, James and Ullman, Daniel},
  title = {The fractional chromatic number of {Mycielski}'s graphs},
  journal = {Journal of Graph Theory},
  volume = {19},
  number = {3},
  pages = {411--416},
  year = {1995},
  doi = {10.1002/jgt.3190190313}
}

@book{ScheinermanUllman2011,
  author = {Scheinerman, Edward R. and Ullman, Daniel H.},
  title = {Fractional Graph Theory: A Rational Approach to the Theory of Graphs},
  publisher = {Dover Publications},
  address = {Mineola, NY},
  year = {2011},
  isbn = {9780486485935},
  note = {Reprint of the 1997 original}
}

@book{Schrijver1986,
  author = {Schrijver, Alexander},
  title = {Theory of Linear and Integer Programming},
  publisher = {John Wiley \& Sons},
  year = {1986}
}

@article{Sperner1928,
  author = {Sperner, Emanuel},
  title = {Ein {Satz} {\"u}ber {Untermengen} einer endlichen {Menge}},
  journal = {Mathematische Zeitschrift},
  volume = {27},
  pages = {544--548},
  year = {1928},
  doi = {10.1007/BF01171114}
}

@article{Tardif2001,
  author = {Tardif, Claude},
  title = {Fractional chromatic numbers of cones over graphs},
  journal = {Journal of Graph Theory},
  volume = {38},
  number = {2},
  pages = {87--94},
  year = {2001},
  doi = {10.1002/jgt.1025}
}
\end{document}